\documentclass[reqno,12pt]{amsart}

\usepackage[active]{srcltx}
\usepackage{color}
\usepackage{amsmath,amsthm,amssymb}
\usepackage{epsfig}
\usepackage{graphicx}
\usepackage{amssymb}
\usepackage{latexsym}
\usepackage{pdfpages}

\usepackage{enumitem}

\usepackage{ulem}

\usepackage{ifpdf}

\ifpdf 
\usepackage[hidelinks]{hyperref}
\else 
\fi

\usepackage[usenames,dvipsnames]{pstricks}
\usepackage{epsfig}
\usepackage{pst-grad} 
\usepackage{pst-plot} 

\usepackage{color}

\newtheorem{theorem}{Theorem}[section]
\newtheorem{lemma}[theorem]{Lemma}
\newtheorem{definition}[theorem]{Definition}
\newtheorem{proposition}[theorem]{Proposition}

\newtheorem{remark}[theorem]{Remark}
\newtheorem{example}[theorem]{Example}
\newtheorem*{theorem*}{\it Theorem}

\numberwithin{equation}{section}

\def\qed{\hfill$\Box$ \vskip.3cm}

\def\1{\raisebox{2pt}{\rm{$\chi$}}}

\def\Xint#1{\mathchoice
	{\XXint\displaystyle\textstyle{#1}}%
	{\XXint\textstyle\scriptstyle{#1}}%
	{\XXint\scriptstyle\scriptscriptstyle{#1}}%
	{\XXint\scriptscriptstyle\scriptscriptstyle{#1}}%
	\!\int}
\def\XXint#1#2#3{{\setbox0=\hbox{$#1{#2#3}{\int}$}
		\vcenter{\hbox{$#2#3$}}\kern-.5\wd0}}

\def\dashint{\Xint-}

\begin{document}
	
\title[Trace space of anisotropic least gradient functions]{\bf The trace space of anisotropic least gradient functions depends on the anisotropy}
	
\author[W. G\'{o}rny]{Wojciech G\'{o}rny}
	
\address{ W. G\'{o}rny: Faculty of Mathematics, Universit\"at Wien, Oskar-Morgerstern-Platz 1, 1090 Vienna, Austria; Faculty of Mathematics, Informatics and Mechanics, University of Warsaw, Banacha 2, 02-097 Warsaw, Poland
\hfill\break\indent
{\tt  wojciech.gorny@univie.ac.at }
}

%
%
	
	
\setcounter{tocdepth}{1}

\date{\today}
	
\begin{abstract}
We study the set of possible traces of anisotropic least gradient functions. We show that even on the unit disk it changes with the anisotropic norm: for two sufficiently regular strictly convex norms the trace spaces coincide if and only if the norms coincide. The example of a function in exactly one of the trace spaces is given by a characteristic function of a suitably chosen Cantor set.
\end{abstract}
	
\maketitle


\section{Introduction}

The least gradient problem is the following minimisation problem
\begin{equation}\label{eq:lgpisotropic}\tag{LGP}
\min \bigg\{ \int_{\Omega} |Du|:  \, u \in BV(\Omega), \, u|_{\partial\Omega} = f \bigg\},
\end{equation}
where $f \in L^1(\partial\Omega)$. It was first considered in this form by Sternberg, Williams and Ziemer in \cite{SWZ}, but its roots go back to the works of Miranda \cite{Mir0,Mir} and Bombieri, de Giorgi and Giusti \cite{BGG} on area-minimising sets. It can be also expressed as the Dirichlet problem for the $1$-Laplace operator, see \cite{MazRoSe}. This problem and its anisotropic versions appear in relation to free material design, conductivity imaging, and optimal transport (see \cite{DS,GRS2017NA,JMN}).

Since problem \eqref{eq:lgpisotropic} consists of minimisation of a linear growth functional, the natural energy space is $BV(\Omega)$ and the trace operator $T: BV(\Omega) \rightarrow L^1(\partial\Omega)$ is not continuous with respect to the weak* convergence, so (unlike the $p$-Laplace equation) existence of solutions does not immediately follow from the use of the direct method of calculus of variations. The authors of \cite{BGG} have shown that for a solution of \eqref{eq:lgpisotropic} the superlevel sets are area-minimising, so it is natural to require that $\Omega \subset \mathbb{R}^N$ is strictly convex (or more generally, that $\partial\Omega$ has nonnegative mean curvature and is not locally area-minimising). In this case, it was shown in \cite{SWZ} that solutions exist for continuous boundary data. Later, this result was extended in \cite{Gor2021IUMJ,Mor} to boundary data which are continuous $\mathcal{H}^{N-1}$-a.e. on $\partial\Omega$. On the other hand, it was shown in \cite{ST} that even when $\Omega$ is a disk, there exists boundary data $f \in L^\infty(\partial\Omega)$ for which there is no solution to problem \eqref{eq:lgpisotropic}: it is given by a certain set on $\partial\Omega$ which is homeomorphic to the Cantor set. Similar examples of this type were also considered in \cite{DosS}, where the author proves that the construction from \cite{ST} can be made on any set with $C^2$ boundary; in \cite{Gor2018CVPDE}, where it was studied in relation to stability results for solutions to \eqref{eq:lgpisotropic}; and in a recent preprint \cite{Kli}, where the author shows that the set of functions on $\partial\Omega$ for which exist solutions to \eqref{eq:lgpisotropic} is not a vector space.

Our main focus in this paper is on existence of solutions to the anisotropic least gradient problem, i.e.
\begin{equation}\label{eq:problem}\tag{aLGP}
\min \bigg\{\int_{\Omega} |Du|_\phi : \, u \in BV(\Omega), \, u|_{\partial\Omega} = f \bigg\},
\end{equation}
where $f \in L^1(\partial\Omega)$. In the study of the anisotropic least gradient problem, the most important special cases are $\phi(x,\xi) = a(x) |\xi|$ (called the weighted least gradient problem) and the case when $\phi$ is a strictly convex norm (i.e. its unit ball is strictly convex). A particular class of metric integrands that we will use throughout the paper are the $l_p$ norms, i.e.
$$ \phi(x,\xi) = l_p(\xi) := \sqrt[p]{|\xi_1|^p + |\xi_2|^p}.$$
Note that the $l_2$ norm is the standard Euclidean norm and $l_1$ is the Manhattan metric. 

The standard assumption used to obtain existence of solutions to problem \eqref{eq:problem} was introduced in \cite{JMN} and is called the {\it barrier condition}. It is a local property at every point $x_0 \in \partial\Omega$, which states that the boundary is not area-minimising with respect to internal variations. Under this assumption, existence of solutions for continuous boundary data was proved in \cite{JMN}. Later, the result was extended to boundary data which are continuous $\mathcal{H}^{N-1}$-almost everywhere, see \cite{Gor2021IUMJ,Mor}. When $\phi$ is a strictly convex norm, the barrier condition is weaker than strict convexity of $\Omega$ and stronger than convexity of $\Omega$, see \cite{Gor2021IUMJ}.

We focus on the case when $\Omega$ is a two-dimensional disk and $\phi$ is a strictly convex norm. The discussion in the previous paragraph implies that there exist solutions to \eqref{eq:problem} for continuous boundary data. Our goal is to show that in spite of this existence result, which is uniform with respect to $\phi$, the set of functions on $\partial\Omega$ for which exist solutions to \eqref{eq:problem} changes with the anisotropic norm. In other words, we study the trace space of anisotropic least gradient functions. We focus on the two-dimensional case and for simplicity we work on the unit ball. The main result of the paper is the following.

\begin{theorem}\label{thm:tracespace}
Suppose that $\phi_1$ and $\phi_2$ are two strictly convex norms of class $C^2$. Unless $\phi_1 = c \phi_2$ for some $c > 0$, there exists a function $f \in L^\infty(\partial\Omega)$ such that there exists a solution to \eqref{eq:problem} for $\phi_1$, but there is no solution to \eqref{eq:problem} for $\phi_2$.
\end{theorem}

In other words, the trace spaces of anisotropic least gradient functions for sufficiently regular $\phi_1$ and $\phi_2$ coincide if and only if $\phi_1 = c \phi_2$. This is achieved using a construction of a suitable set on $\partial\Omega$ which is homeomorphic to the Cantor set, similarly to the examples appearing in \cite{DosS,Gor2018CVPDE,Kli,ST}, with parameters carefully chosen so that certain key inequalities are satisfied on every level of the construction. Note that this phenomenon is closely related to the fact that we minimise a functional with linear growth; it does not appear for the anisotropic $p$-Laplace equation. Also, it is not related to the regularity of the boundary data, as the functions which lie in exactly one of the trace spaces will be obtained using variants of the same construction.

The structure of the paper is as follows. In Section \ref{sec:preliminaries}, we recall the required notions on anisotropic BV spaces and the anisotropic least gradient problem. The rest of the paper is devoted to the proof of the main result (Theorem \ref{thm:tracespace}). In Section \ref{sec:restriction}, we argue that we can restrict our attention to the case when both the boundary datum and the solutions are characteristic functions of some sets and present a simple example on a non-stricly convex domain, which serves as a toy model of our reasoning in the proof of the main result. In Section \ref{sec:notation}, we introduce the notation for the main part of the proof of Theorem \ref{thm:tracespace}. The proofs of most results are located in Section \ref{sec:proof}. We construct of a set $F_\infty$ which is homeomorphic to the Cantor set, and give two general conditions \eqref{eq:maininequality} and \eqref{eq:mainequality}, under which we can find structure of solutions to approximate problems at every stage of the construction, from which follows respectively nonexistence or existence of solutions in the limit. We prove Theorem \ref{thm:tracespace} by showing that we can choose $F_\infty$ so that one of the conditions holds for $\phi_1$ and the other for $\phi_2$. We complement the result with a short discussion on nonsmooth and non-strictly convex norms.

\section{Preliminaries}\label{sec:preliminaries}

In this Section, we shortly recall the main definitions and results related to the anisotropic least gradient problem. From now on, we assume that $\Omega \subset \mathbb{R}^2$ is an open bounded set with Lipschitz boundary and $\phi$ denotes a norm on $\mathbb{R}^2$. We will often require that $\phi$ is strictly convex, i.e. its unit ball is strictly convex. We focus on results related to existence of solutions and briefly discuss the assumptions on the domain and regularity of boundary data. To simplify the presentation, we restrict ourselves to the two-dimensional case.

First, we recall the notion of anisotropic BV spaces (see \cite{AB}); a classical reference for the general theory in the isotropic case is \cite{AFP}. We start with the definition of the anisotropic total variation.

\begin{definition}
The $\phi-$total variation of $u \in L^1(\Omega)$ is defined by the formula
\begin{equation}
\int_\Omega |Du|_\phi = \sup \, \bigg\{ \int_\Omega u \, \mathrm{div} (\mathbf{z}) \, dx : \, \phi^0(\mathbf{z}(x)) \leq 1 \, \, \, \text{a.e.},\,\, \mathbf{z} \in C_c^1(\Omega)  \bigg\},
\end{equation}
where $\phi^0: \mathbb{R}^2 \rightarrow [0, \infty)$ given by the formula
\begin{equation}
\phi^0 (\xi^*) = \sup \, \{ \langle \xi^*, \xi \rangle : \, \xi \in \mathbb{R}^2, \, \phi(\xi) \leq 1 \}.
\end{equation}
is the polar function of $\phi$.

We say that $u \in BV_\phi(\Omega)$ if its $\phi-$total variation is finite. Since any norm on $\mathbb{R}^2$ is equivalent to the Euclidean norm, we have
$$\lambda \int_\Omega |Du| \leq \int_\Omega |Du|_\phi \leq \Lambda \int_\Omega |Du|,$$
so $BV_\phi(\Omega) = BV(\Omega)$ as sets. They are equipped with different (but equivalent) norms. Furthermore, given a measurable set $E \subset \Omega$, we define its $\phi-$perimeter as
$$P_\phi(E, \Omega) = \int_{\Omega} |D\chi_E|_\phi.$$
If $P_\phi(E, \Omega) < \infty$, we say that $E$ is a set of finite $\phi-$perimeter in $\Omega$. Again, since any norm is equivalent to the Euclidean norm, sets of finite perimeter coincide with sets of finite $\phi$-perimeter.
\end{definition}

Now, we recall the notion of $\phi$-least gradient functions. Historically, this was how the problem was first introduced (in the isotropic case) in \cite{BGG,Mir0,Mir}.

\begin{definition}
We say that $u \in BV(\Omega)$ is a function of $\phi$-least gradient, if for all $v \in BV(\Omega)$ with compact support (equivalently: with zero trace) we have
$$ \int_\Omega |Du|_\phi \leq \int_\Omega |D(u+v)|_\phi.$$
\end{definition}

The first two results concern local properties of $\phi$-least gradient functions. The first one states that a limit of $\phi$-least gradient functions is itself a $\phi$-least gradient function; while the proof in \cite{Mir} is given in the isotropic case, the proof uses only basic properties of BV functions and a generalisation to the anisotropic case is straightforward.

\begin{theorem}\label{thm:miranda}
Suppose that $u_n \in BV(\Omega)$ is a sequence of $\phi$-least gradient functions which converges in $L^1(\Omega)$ to $u \in BV(\Omega)$. Then, $u$ is a $\phi$-least gradient function.
\end{theorem}

The second result concerns the superlevel sets of $\phi$-least gradient functions. It is a consequence of the co-area formula, and it was first proved in \cite{BGG} in the isotropic case and in \cite{Maz} in the anisotropic case.

\begin{theorem}\label{thm:bgg}
Suppose that $u \in BV(\Omega)$ is a $\phi$-least gradient function. Then, for all $t \in \mathbb{R}$, the function $\chi_{\{ u \geq t \}}$ is also a function of $\phi-$least gradient.
\end{theorem}

In two dimensions, whenever $\phi$ is strictly convex, the only connected $\phi$-minimal surfaces are line segments. Therefore, in this case Theorem \ref{thm:bgg} implies that the boundary of superlevel sets of every $\phi$-least gradient function is a locally finite union of line segments.

In this language, the anisotropic least gradient problem \eqref{eq:problem} consists of finding a $\phi$-least gradient function with a prescribed trace. We summarise the discussion in the introduction in the following result (see \cite{Gor2021IUMJ,Mor}).

\begin{theorem}\label{thm:existence}
Suppose that $\Omega$ is strictly convex. Let $\phi$ be a strictly convex norm and suppose that $f \in L^1(\partial\Omega)$ is continuous $\mathcal{H}^{N-1}$-almost everywhere on $\partial\Omega$. Then, there exists a solution to problem \eqref{eq:problem}.
\end{theorem}

As a particular case, whenever $\Omega$ is strictly convex and $f \in BV(\partial\Omega)$ is a (one-dimensional) function of bounded variation, there exists a solution to problem \eqref{eq:problem}. As a consequence, in order to prove Theorem \ref{thm:tracespace}, we need to consider boundary data with very low regularity.

\section{Construction}

In this Section, we prove the main result of the paper, i.e. Theorem \ref{thm:tracespace}. First, we start with some motivations for the construction; in Section \ref{sec:restriction}, we argue that we can restrict our attention to characteristic functions and present a simple example on a non-strictly convex domain. Then, in Section \ref{sec:notation} we introduce the notation used in the proofs in Section \ref{sec:proof}.

\subsection{Basic idea}\label{sec:restriction}

The first result is a simple exercise in the theory of BV functions.

\begin{lemma}\label{lem:superlevelsets}
Let $u \in BV(\Omega)$ and $f \in L^1(\partial\Omega)$. The following conditions are equivalent:

(1) $Tu = f$;

(2) For all but countably many $t \in \mathbb{R}$, we have $T\chi_{\{ u \geq t \}} = \chi_{ \{ f \geq t \} }$;

(3) For almost all $t \in \mathbb{R}$, we have $T\chi_{\{ u \geq t \}} = \chi_{ \{ f \geq t \} }$.
\end{lemma}

\begin{proof}
The implication $(1) \Rightarrow (2)$ was shown in \cite[Lemma 2.12]{Gor2020PAMS}. The implication $(2) \Rightarrow (3)$ is clear, since a countable set is of zero $\mathcal{L}^1$-measure. For the implication $(3) \Rightarrow (1)$, suppose that $Tu = g \in L^1(\partial\Omega)$. Then, by the implication $(1) \Rightarrow (2)$, we have that for all but countably many $t \in \mathbb{R}$ we have $T\chi_{\{ u \geq t \}} = \chi_{\{ g \geq t \} }$. Hence, for almost all $t$ we have $\{ f \geq t \} = \{ g \geq t \}$, so $f = g$ a.e.
\end{proof}

The next two results show that in order to describe solutions to \eqref{eq:problem} in the case when the boundary datum is a characteristic function of some set, it is sometimes enough to look for solutions which are characteristic functions themselves.

\begin{lemma}\label{lem:characteristicfunctions}
Suppose that $f = \chi_F$, where $F \subset \partial\Omega$. Let $u \in BV(\Omega)$ be a solution to problem \eqref{eq:problem}. Then, there exists a set $E \subset \Omega$ of finite perimeter such $\chi_E$ is a solution to problem \eqref{eq:problem}. 
\end{lemma}

\begin{proof}
By Theorem \ref{thm:bgg}, for all $t \in \mathbb{R}$ the function $\chi_{\{ u \geq t \}}$ is a function of $\phi$-least gradient. But, by Lemma \ref{lem:superlevelsets}, for all but countably many $t \in (0,1)$ we have $T\chi_{\{ u \geq t \}} = \chi_{\{ \chi_F \geq t \} } = \chi_F$. Therefore, it suffices to take such $t \in (0,1)$ and set $E = \{ u \geq t \}$.
\end{proof}

\begin{lemma}\label{lem:uniquesolution}
Let $f = \chi_F$, where $F \subset \partial\Omega$. Suppose that there exists exactly one set $E \subset \Omega$ of finite perimeter $\chi_E$ which is a solution of \eqref{eq:problem}. Then, if $u \in BV(\Omega)$ is a solution to \eqref{eq:problem}, we have $u = \chi_E$.
\end{lemma}

\begin{proof}
By Lemma \ref{lem:superlevelsets}, for all but countably many $t \in \mathbb{R}$ we have $T\chi_{\{ u \geq t \}} = \chi_{\{ f \geq t \}}$. On the other hand, by Theorem \ref{thm:bgg}, for all $t \in \mathbb{R}$ the function $\chi_{\{ u \geq t \}}$ is a function of $\phi$-least gradient. We consider three cases. For $t > 1$ we have $T\chi_{\{ u \geq t \}} = \chi_{\{ f \geq t \}} = 0$, but the only $\phi$-least gradient function with this trace is constant and equal to zero, so $\{ u \geq t \}$ is a set of zero measure. For $t < 0$ we have $T\chi_{\{ u \geq t \}} = \chi_{\{ f \geq t \}} = 1$, but again the only $\phi$-least gradient function with this trace is constant and equal to one, so $\{ u \geq t \}$ is a set of full measure. Finally, for $t \in (0,1)$ we have $T\chi_{\{ u \geq t \}} = \chi_{\{ f \geq t \}} = \chi_F$, and by our assumption we have $\{ u \geq t \} = E$. In this way, we prescribed the superlevel sets for almost all $t \in \mathbb{R}$ and we obtain $u= \chi_E$.
\end{proof}

Therefore, if the boundary datum is a characteristic function, we may only look for solutions which are also characteristic functions. Let us note that the proofs of every Lemma only used basic properties of BV functions and functions of $\phi$-least gradient, so they are also valid for other anisotropies and in higher dimensions.

Now, we present a very simple example, which will serve as a motivation for the construction in Sections \ref{sec:notation} and \ref{sec:proof}. We will show that given two anisotropic norms $\phi_1$ and $\phi_2$, without the assumption of strict convexity of $\Omega$ it can be very easy to construct a boundary datum such that the problem is solvable for $\phi_1$ and it is not solvable for $\phi_2$. As we will see in the next Sections, the case when $\Omega$ is strictly convex is more difficult to handle. This is due to Theorem \ref{thm:existence}; counterexamples to existence are necessarily more refined.

\begin{example}
Let $\Omega = [0,1]^2$. Our goal is to find a boundary datum $f$ for which there is no solution to \eqref{eq:problem} in the isotropic case, but there exists a norm for which there is a solution to \eqref{eq:problem}. Fix four points on $\partial\Omega$: let $p_{(0,0)} = (0,a)$, { $p_{(0,1)} = (b,0)$}, $p_{(1,0)} = (1-b,0)$ and $p_{(1,1)} = (1,a)$, where { $a \in (0,1)$ and $b \in (0,\frac12)$}. For $i = 0,1$, denote by $\Gamma_i$ the arc on $\partial\Omega$ from $p_{(i,0)}$ to $p_{(i,1)}$ (we mean the shorter of two such arcs on $\partial\Omega$), and by $\ell_i$ the line segment from $p_{(i,0)}$ to $p_{(i,1)}$. We denote by $\ell'_0$ the line segment from $p_{(0,0)}$ to $p_{(1,1)}$ and by $\ell'_1$ the line segment from $p_{(0,1)}$ to $p_{(1,0)}$. Finally, we set $F = \Gamma_1 \cup \Gamma_2$ and $f= \chi_F$.

Take $\phi_1 = l_2$ and $\phi_2 = l_3$. The norms $\phi_1$ and $\phi_2$ are strictly convex, all connected minimal surfaces are line segments, so we need to consider only two competitors. The first one (denoted by $E_1$) is the set with two connected components, which are triangles whose boundary (in $\mathbb{R}^2$) is composed of the line segments $\ell_i$ and the arcs $\Gamma_i$. The second one (denoted by $E_2$) is the rectangle whose boundary is composed of the line segments $\ell'_0$, $\ell'_1$ and the arcs $\Gamma_i$ (in which case the trace condition is violated, so a solution does not exist).

Then, the $l_2$-perimeter (i.e. the Euclidean perimeter) of the $E_1$ in $\Omega$ equals $2\sqrt{a^2 + b^2}$ and, with a slight abuse of notation, the $l_2$-perimeter of $E_2$ in $\Omega$ is $2 - 2b$ (technically, since $\ell'_1$ is a subset of $\partial\Omega$, only $\ell'_0$ enters the calculation of the $l_2$-perimeter of $E_2$ in $\Omega$; however, for the purpose of comparison with $E_1$, we also count the length of the line segment $\ell'_1$ on which $E_2$ violates the trace condition). On the other hand, the $l_3$-perimeter of $E_1$ equals $2\sqrt[3]{a^3 + b^3}$ and the $l_3$-perimeter of $E_2$ is again $2 - 2b$. We choose the parameters $a,b$ so that we get different minimal sets for the $l_2$- and $l_3$-perimeters. For instance, choose $a = \frac{1}{2}$ and $b = \frac{2}{5}$. Then, for the $l_3$-perimeter the (unique) solution to \eqref{eq:problem} is $E_1$, but for the $l_2$-perimeter, there is no solution.
\end{example}

\subsection{Notation}\label{sec:notation}

Let us introduce notation for the remainder of the paper. Set $\Omega = B(0,1) \subset \mathbb{R}^2$. Below, we presents a general framework for the construction of a set with positive $\mathcal{H}^1$-measure on $\partial\Omega$, which is homeomorphic to the Cantor set. We will use a particular case of this construction to find a boundary datum, for which (for some given strictly convex norms $\phi_1, \phi_2$ of class $C^2$) there exists a solution to problem \eqref{eq:problem} for $\phi_1$, but there is no solution for $\phi_2$. In relation to the least gradient problem, the first example of this type appeared in \cite{ST} (for other occurrences, see \cite{DosS,Gor2018CVPDE,Kli}).

Denote by $\alpha$ the angular coordinate on $\partial B(0,1)$. Denote the origin by $q$. The zeroth step of the construction is as follows: we fix $\alpha_0 \in (0,\frac{\pi}{2})$ and take two points $p_0, p_1$ such that the angle $p_0 q p_1$ is equal to $\alpha_0$. We denote the arc on $\partial\Omega$ from $p_0$ to $p_1$ by $F_0$ (here and in the whole construction, we always mean the shorter of the two arcs on the boundary, { and we assume that the arcs are closed}). Moreover, we denote by $E_0$ the { open} bounded set whose boundary is composed of the arc $F_0$ and the line segment $\overline{p_0 p_1}$ { (with endpoints)}.

For the first step of the construction, we fix $\alpha_1 \in (0,\frac{\alpha_0}{2})$. Then:

\begin{enumerate}

\item Rename the points: denote $p_{(0,0)} = p_0$ and $p_{(1,1)} = p_1$;

\item Add two more points: we choose $p_{(0,1)}, p_{(1,0)} \in F_0$ so that the angles $p_{(0,0)} q p_{(0,1)}$ and $p_{(1,0)} q p_{(1,1)}$ are equal to $\alpha_1$;

\item Denote the arc on $\partial\Omega$ from $p_{(0,0)}$ to $p_{(0,1)}$ by $\Gamma_{(0)}$. Similarly, we denote the arc on $\partial\Omega$ from $p_{(1,0)}$ to $p_{(1,1)}$ by $\Gamma_{(1)}$;

\item Denote by $\ell_{(0)}$ the line segment $\overline{p_{(0,0)} p_{(0,1)}}$ and by $\ell_{(1)}$ the line segment $\overline{p_{(1,0)} p_{(1,1)}}$;

\item Denote by $\Delta_{(0)}$ the { open} bounded set whose boundary is composed of $\Gamma_{(0)}$ and $\ell_{(0)}$. Similarly, the { open} bounded set whose boundary is composed of $\Gamma_{(1)}$ and $\ell_{(1)}$ is denoted by $\Delta_{(1)}$;

\item Set $F_1 = \Gamma_{(0)} \cup \Gamma_{(1)}$; then, we have $F_1 \subset F_0$.

\item Set $E_1 = \Delta_{(0)} \cup \Delta_{(1)}$; observe that $E_1 \subset E_0$.

\end{enumerate}

Now, we present the $n$-th step of the construction ($n \geq 2$). We denote points in the boundary of the set $F_{n-1}$ using a binary sequence $(m_1,...,m_n)$. For all $(m_1,...,m_n) \in \{ 0, 1 \}^{n}$, fix $\alpha_{(m_1,...,m_n)} \in (0,\frac{\alpha_{(m_1,...,m_{n-1})}}{2})$. Then:

\begin{enumerate}

\item Rename the $2^n$ points $p_{(m_1,...,m_n)}$ from the previous step. In the following way: depending on the value of $m_n$, we set
$$ p_{(m_1,...,m_{n-1},0,0)} := p_{(m_1,...,m_{n-1},0)}$$
and
$$ p_{(m_1,...,m_{n-1},1,1)} := p_{(m_1,...,m_{n-1},1)};$$

\item Add $2^n$ more points: we choose $p_{(m_1,...,m_{n-1},0,1)}, p_{(m_1,...,m_{n-1},1,0)} \in \Gamma_{(m_1,...,m_{n-1})}$ so that the angle $p_{(m_1,...,m_{n-1},0,0)} q p_{(m_1,...,m_{n-1},0,1)}$ is equal to $\alpha_{(m_1,...,m_{n-1},0)}$ and the angle $p_{(m_1,...,m_{n-1},1,0)} q p_{(m_1,...,m_{n-1},1,1)}$ is equal to $\alpha_{(m_1,...,m_{n-1},1)}$;

\item Denote the arc on $\partial\Omega$ from $p_{(m_1,...,m_n,0)}$ to $p_{(m_1,...,m_n,1)}$ by $\Gamma_{(m_1,...,m_n)}$;

\item Denote by $\ell_{(m_1,...,m_n)}$ the line segment $\overline{p_{(m_1,...,m_n,0)} p_{(m_1,...,m_n,1)}}$;

\item Denote by $\Delta_{(m_1,...,m_n)}$ the { open} bounded set whose boundary is 
composed of $\Gamma_{(m_1,...,m_n)}$ and $\ell_{(m_1,...,m_n)}$;

\item Set
$$F_n = \bigcup_{(m_1,...,m_n) \in \{ 0,1 \}^n} \Gamma_{(m_1,...,m_n)}.$$
Then, { the sets $F_n$ are closed and} we have $F_n \subset F_{n-1}$.

\item Set 
$$E_n = \bigcup_{(m_1,...,m_n) \in \{ 0,1 \}^n} \Delta_{(m_1,...,m_n)}.$$
Then, { the sets $E_n$ are open and} we have $E_n \subset E_{n-1}$.
\end{enumerate}

We observe that 
$$\bigcap_{n = 0}^\infty E_n = \emptyset$$
and that
\begin{equation}
F_\infty = \bigcap_{n=0}^{\infty} F_n
\end{equation}
is homeomorphic to the Cantor set. Depending on the choice of the family of angles $\alpha_{(m_1,...,m_n)}$, it may have zero or positive $\mathcal{H}^1$-measure.

This ends the construction of the Cantor set $F_\infty$, whose characteristic function will be the boundary datum. In the process, we also constructed approximations $F_n$ of the boundary datum and sets $E_n$ whose characteristic function is a candidate for a solution for boundary data $\chi_{F_n}$. In order to describe another candidate for a solution for boundary data $\chi_{F_n}$, we introduce the following additional notation.

Set { $E'_0 = E_0 \cup (\overline{p_0 p_1} \setminus \{ p_0, p_1 \})$}. Denote by $\Delta'$ the { open} bounded set whose boundary is composed of { the} arc on $\partial\Omega$ from $p_{(0,1)}$ to $p_{(1,0)}$ and the line segment $\overline{p_{(0,1)} p_{(1,0)}}$. Then, set $E'_1 = E'_0 \setminus \Delta'$. From now on, let $n \geq 2$.

\begin{enumerate}\setcounter{enumi}{7}

\item Denote by $\Gamma'_{(m_1,...,m_{n-1})}$ the arc on $\partial\Omega$ from $p_{(m_1,...,m_{n-1},0,1)}$ to $p_{(m_1,...,m_{n-1},1,0)}$;

\item Denote by $\ell'_{(m_1,...,m_{n-1})}$ the line segment $\overline{p_{(m_1,...,m_{n-1},0,1)} p_{(m_1,...,m_{n-1},1,0)}}$;

\item Denote by $\Delta'_{(m_1,...,m_{n-1})}$ the { open} bounded set whose boundary is composed of $\Gamma'_{(m_1,...,m_{n-1})}$ and $\ell'_{(m_1,...,m_{n-1})}$;

\item Set
$$ E'_n = E'_{n-1} \setminus \bigg( \bigcup_{(m_1,...,m_{n-1}) \in \{ 0,1 \}^{n-1}} \Delta'_{(m_1,...,m_{n-1})} \bigg).$$
Then, { the sets $E'_n$ are closed relative to $\Omega$,} we have $E'_n \subset E'_{n-1}$, and the intersection $E'_\infty = \bigcap_{n=0}^\infty E'_n$ { is closed relative to $\Omega$ and} has positive Lebesgue measure.

\end{enumerate}

Moreover, we may require that at every step of the construction the angles $\alpha_{(m_1,...,m_{n-1},0)}$ and $\alpha_{(m_1,...,m_{n-1},1)}$ are equal, so that the line segments $\ell_{(m_1,...,m_n)}$ and $\ell'_{(m_1,...,m_n)}$ are parallel. Whenever we fix $(m_1,...,m_{n-1}) \in \{ 0, 1 \}^{n-1}$, we denote the common value by $\alpha_{(m_1,...,m_n)}$. From now on, we consider only such configurations. The situation is presented in Figures \ref{fig:notation} and \ref{fig:notationpart2}. Both Figures contain the same configuration of points; in the first one we highlight the notation for line segments and arcs between the four points, and in the second one we present the sets $\Delta_{(m_1,...,m_n)}$ and $\Delta'_{(m_1,...,m_{n-1})}$.

\begin{figure}
    \centering
    \includegraphics[scale=0.4]{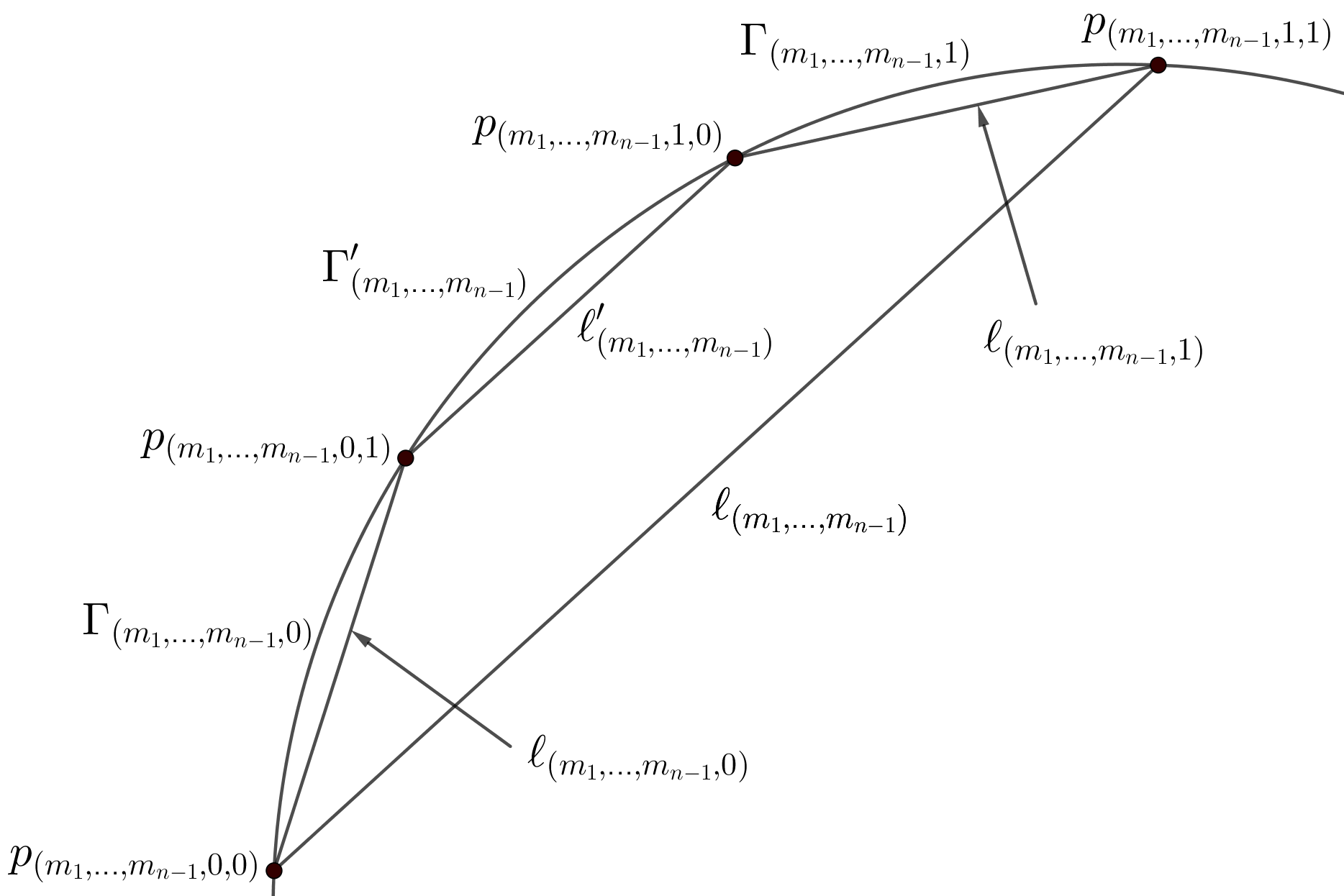}
    \caption{Notation - part one}
    \label{fig:notation}
\end{figure}

\begin{figure}
    \centering
    \includegraphics[scale=0.4]{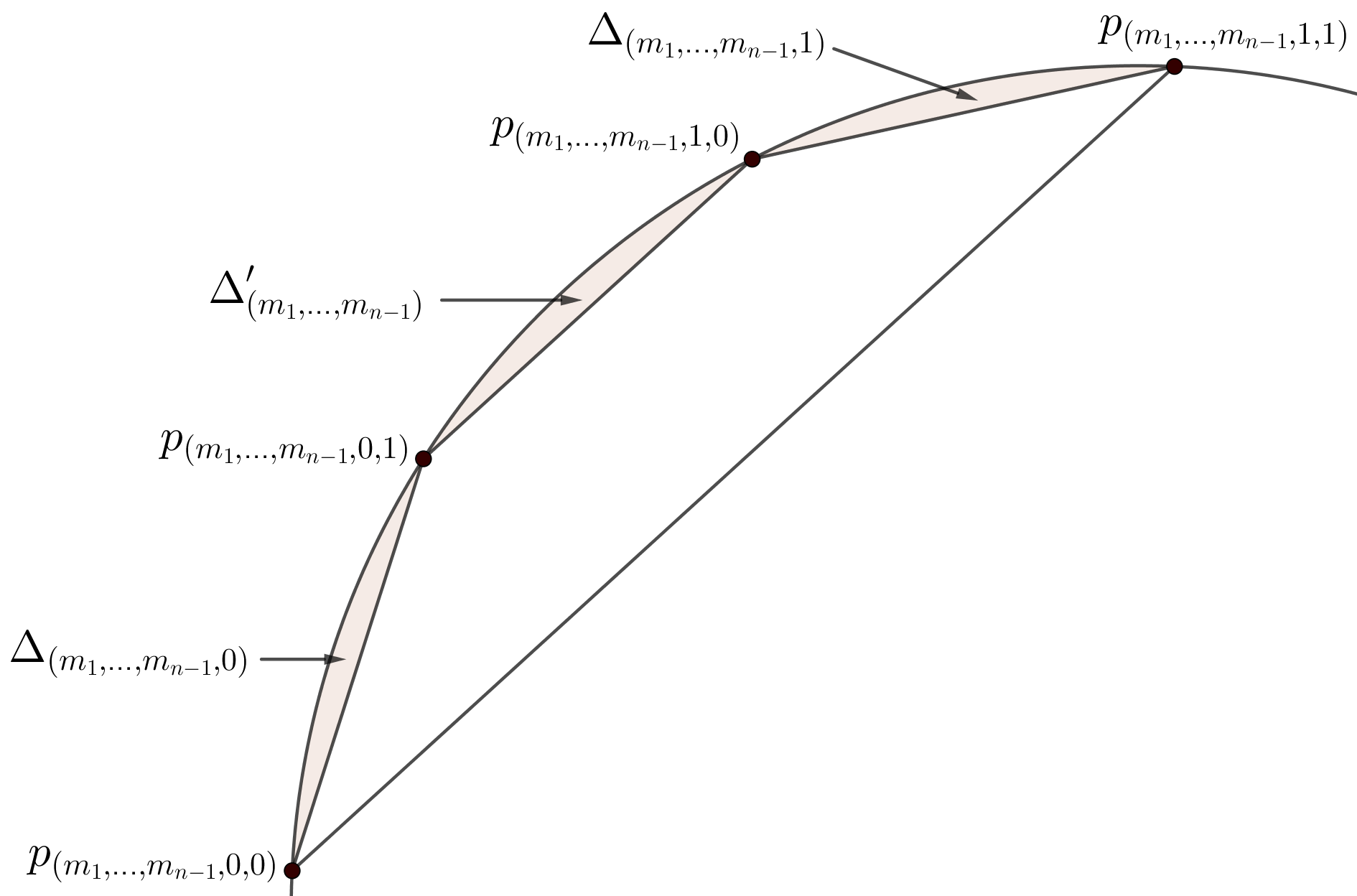}
    \caption{Notation - part two}
    \label{fig:notationpart2}
\end{figure}

Finally, to simplify the notation, for a line segment $\ell = \overline{xy} \subset \overline{\Omega}$ we denote by $\| \ell \|_\phi$ its anisotropic length, i.e.
\begin{equation}
\| \ell \|_\phi = \phi(y-x).
\end{equation}
We will typically apply this to the line segments $\ell_{(m_1,...,m_n)}$ or $\ell'_{(m_1,...,m_{n-1})}$.

The core of our strategy will be to consider trapezoids, which arise during the $n$-th step of the construction of the Cantor set, and check which configurations are optimal in such a trapezoid. To be exact, consider the trapezoid whose vertices are points $p_{(m_1,...,m_{n-1},0,0)}$, $p_{(m_1,...,m_{n-1},0,1)}$, $p_{(m_1,...,m_{n-1},1,0)}$ and $p_{(m_1,...,m_{n-1},1,1)}$. Its sides are the line segments $\ell_{(m_1,...,m_{n-1})}$, $\ell_{(m_1,...,m_{n-1},0)}$, $\ell_{(m_1,...,m_{n-1},1)}$ and $\ell'_{(m_1,...,m_{n-1})}$. Fix a strictly convex norm $\phi$ and denote
\begin{equation}\label{eq:definitionofh}
\begin{split}
&h_\phi(m_1,...,m_{n-1}) := \\
&:= (\| \ell_{(m_1,...,m_{n-1})} \|_\phi + \| \ell'_{(m_1,...,m_{n-1})} \|_\phi) - (\| \ell_{(m_1,...,m_{n-1},0)} \|_\phi + \| \ell_{(m_1,...,m_{n-1},1)} \|_\phi).
\end{split}
\end{equation}
Note that for fixed $(m_1,...,m_{n-1}) \in \{ 0, 1 \}^{n-1}$ the value of $h_\phi(m_1,...,m_{n-1})$ is fully determined by $\alpha_{(m_1,...,m_{n-1})}$ and $\alpha_{(m_1,...,m_n)}$. The construction as described above is quite flexible, and in the results below we will often require that one of the two conditions hold: either 
\begin{equation}\label{eq:maininequality}
\text{For all } n \in \mathbb{N} \text{ and } (m_1,...,m_{n-1}) \in \{ 0 , 1 \}^{(n-1)} \text{ we have }
h_\phi(m_1,...,m_{n-1}) > 0;
\end{equation}
or
\begin{equation}\label{eq:mainequality}
\text{For all } n \in \mathbb{N} \text{ and } (m_1,...,m_{n-1}) \in \{ 0 , 1 \}^{(n-1)} \text{ we have }
h_\phi(m_1,...,m_{n-1}) = 0.
\end{equation}
Clearly, both options are possible within the framework of the construction: if the angles $\alpha_{(m_1,...,m_n)}$ are sufficiently small, the sides of the trapezoid are shorter than the bases, so \eqref{eq:maininequality} holds. On the other hand, since $\phi$ is strictly convex, if the angles $\alpha_{(m_1,...,m_n)}$ are sufficiently close to $\frac{\alpha_{(m_1,...,m_{n-1})}}{2}$, by the triangle inequality we have $h_\phi(m_1,...,m_{n-1}) < 0$. By continuity of $h_\phi(m_1,...,m_{n-1})$, it is possible to choose it in such a way that \eqref{eq:mainequality} holds. From now on, unless specified otherwise, we consider only strictly convex norms.

\begin{remark}
In the case when $\phi$ is the Euclidean norm, the above construction can be simplified. For instance, we may use a sequence $\alpha_n$ in place of the family $\alpha_{(m_1,...,m_n)}$, and simply set $\alpha_{(m_1,...,m_n)} := \alpha_n$. Moreover, the function $h_{l_2}$ does not depend on the choice of $(m_1,...,m_{n-1})$, but it is only a function of $\alpha_{n-1}$ and $\alpha_n$, because the relative positions of the four points depend only on these two numbers.
\end{remark}

\subsection{Main results}\label{sec:proof}

Throughout the rest of the paper, we use the notation from the previous subsection.

\begin{proposition}\label{prop:existenceofEn}
Suppose that \eqref{eq:maininequality} holds. Then, for all $n \in \mathbb{N}$, there exists a unique solution $u_n \in BV(\Omega)$ to problem \eqref{eq:problem} with boundary data $f_n = \chi_{F_n}$. It is of the form $u_n = \chi_{E_n}$.
\end{proposition}

Geometrically, the inequality $h_\phi(m_1,...,m_{n-1}) > 0$ means that the sum of the lengths (weighted by $\phi$) of the bases of the trapezoid whose boundary consists of $\ell_{(m_1,...,m_{n-1})}$, $\ell_{(m_1,...,m_{n-1},0)}$, $\ell_{(m_1,...,m_{n-1},1)}$ and $\ell'_{(m_1,...,m_{n-1})}$ is larger than the sum of the lengths (weighted by $\phi$) of its sides. This is required at every step of the construction. 

\begin{proof}
Let $n = 0$. Clearly, the solution $u_0 \in BV(\Omega)$ to problem \eqref{eq:problem} with boundary data $f_0$ is unique and given by $u_0 = \chi_{E_0}$. Also, for $n = 1$, by inequality \eqref{eq:maininequality} the solution $u_1 \in BV(\Omega)$ to problem \eqref{eq:problem} with boundary data $f_1$ is unique and given by $u_1 = \chi_{E_1}$.

Now, take any $n \geq 2$. Suppose that the solution to problem \eqref{eq:problem} with boundary data $f_{n-1}$ is unique and given by $u_{n-1} = \chi_{E_{n-1}}$. By Theorem \ref{thm:existence}, there exists a solution to problem \eqref{eq:problem} with boundary data $f_n$. By Lemma \ref{lem:characteristicfunctions}, there also exists a solution of the form $u_n = \chi_{E}$, where $E \subset \Omega$ is a set of finite perimeter. 

Observe that the function $\chi_{E_{n-1} \cup E}$ is admissible in problem \eqref{eq:problem} with boundary data $f_{n-1}$, since $T \chi_{E_{n-1} \cup E} = T \max(\chi_{E_{n-1}},\chi_{E}) = \max(T_{\chi_{E_{n-1}}}, T_{\chi_{E}}) = \chi_{F_{n-1}}$. Similarly, $\chi_{E_{n-1} \cap E}$ is admissible in problem \eqref{eq:problem} with boundary data $f_{n}$, since $T \chi_{E_{n-1} \cap E} = T \min(\chi_{E_{n-1}},\chi_{E}) = \min(T_{\chi_{E_{n-1}}}, T_{\chi_{E}}) = \chi_{F_{n}}$. By an anisotropic version of \cite[Proposition 3.38]{AFP} (the result is given in the isotropic case, but the proof in the anisotropic case runs along the same lines) we have
$$ P_\phi(E \cap E_{n-1}, \Omega) + P_\phi(E \cup E_{n-1}, \Omega) \leq P_\phi(E,\Omega) + P_\phi(E_{n-1}, \Omega), $$
so both $E_{n-1} \cap E$ and $E_{n-1} \cup E$ are solutions to the respective problems. { Since $\chi_{E_{n-1}}$ is the unique solution to \eqref{eq:problem} with boundary data $f_{n-1}$}, we have that $E_{n-1} = E_{n-1} \cup E$ { up to a set of measure zero}. Therefore, $E \subset E_{n-1}$ { up to a set of measure zero}, so we only need to check how the set $E$ looks like in each of the connected components of $E_{n-1}$, i.e. the sets $\Delta_{(m_1,...,m_{n-1})}$. But then, inequality \eqref{eq:maininequality} implies that $E = E_n$ { up to a set of measure zero}, and by Lemma \ref{lem:uniquesolution} every solution $u_n$ to problem \eqref{eq:problem} is equal to $\chi_{E_n}$.
\end{proof}

Now, we show that such inequality at every step of the construction implies that there is no solution for the anisotropic least gradient problem with boundary data $f =\chi_{F_\infty}$.

\begin{proposition}\label{prop:nonexistence}
Suppose that \eqref{eq:maininequality} holds. Then, provided that $\mathcal{H}^1(F_\infty) > 0$, there is no solution to problem \eqref{eq:problem} with boundary data $f =\chi_{F_\infty}$.
\end{proposition}

\begin{proof}
Suppose that there exists a solution $u \in BV(\Omega)$ to problem \eqref{eq:problem} with boundary data $f_\infty$. By Lemma \ref{lem:characteristicfunctions}, we can assume it is of the form $u = \chi_{E}$. Similarly to the previous proof, observe that the function $\chi_{E \cup E_n}$ is admissible in problem \eqref{eq:problem} with boundary data $f_n$, since $T \chi_{E \cup E_n} = T \max(\chi_{E},\chi_{E_n}) = \max(T_{\chi_{E}}, T_{\chi_{E_n}}) = \chi_{F_n}$. Similarly, $\chi_{E_{n} \cap E}$ is admissible in problem \eqref{eq:problem} with boundary data $f_{n}$, since $T \chi_{E_{n} \cap E} = T \min(\chi_{E_{n}},\chi_{E}) = \min(T_{\chi_{E_{n}}}, T_{\chi_{E}}) = \chi_{F_{n}}$. But, again using an anisotropic version of \cite[Proposition 3.38]{AFP}, we get
$$ P_\phi(E \cap E_n, \Omega) + P_\phi(E \cup E_n, \Omega) \leq P_\phi(E,\Omega) + P_\phi(E_n, \Omega), $$
so both $E_{n} \cap E$ and $E_{n} \cup E$ are solutions to the respective problems. { Since $\chi_{E_{n}}$ is the unique solution to \eqref{eq:problem} with boundary data $f_{n}$}, we have that $E_n = E \cup E_n$ { up to a set of measure zero}. Therefore, $E \subset E_n$ { up to a set of measure zero}. But then, since $\bigcap_{n=0}^\infty E_n = \emptyset$, the set $E$ has zero measure, so $\chi_E$ violates the trace condition, a contradiction.
\end{proof}

Now, we turn our interest to the situation when the condition \eqref{eq:mainequality} holds. Geometrically, the equation $h_\phi(m_1,...,m_{n-1}) = 0$ means that the sum of the lengths (weighted by $\phi$) of the bases of the trapezoid whose boundary consists of $\ell_{(m_1,...,m_{n-1})}$, $\ell_{(m_1,...,m_{n-1},0)}$, $\ell_{(m_1,...,m_{n-1},1)}$ and $\ell'_{(m_1,...,m_{n-1})}$ is equal to the sum of the lengths (weighted by $\phi$) of its sides. This is assumed at every step of the construction. We show that in this case, regardless of the choice of the norm $\phi$, the resulting set $F_\infty$ has positive $\mathcal{H}^1$-measure.

\begin{proposition}\label{prop:positivemeasure}
Suppose that \eqref{eq:mainequality} holds. Then, we have $\mathcal{H}^{1}(F_\infty) > 0$.
\end{proposition}

\begin{proof}
Fix $(m_1,...,m_{n-1}) \in \{ 0, 1 \}^{n-1}$. Recall that in order to compute $h_\phi(m_1,...,m_{n-1})$, we consider the trapezoid with sides $l_{(m_1,...,m_{n-1})}$, $l_{(m_1,...,m_{n-1},0)}$, $l_{(m_1,...,m_{n-1},1)}$ and $l'_{(m_1,...,m_{n-1})}$. Denote by $\nu$ the direction of the (parallel) line segments $l_{(m_1,...,m_{n-1})}$ and $l'_{(m_1,...,m_{n-1})}$, by $\nu_0$ the direction of the line segment $l_{(m_1,...,m_{n-1},0)}$ and by $\nu_1$ the direction of the line segment $l_{(m_1,...,m_{n-1},1)}$. Notice that for a circle of radius one the Euclidean length of the chord corresponding to angle $\alpha$ equals $2 \sin(\frac{\alpha}{2})$. Since $\alpha_{(m_1,...,m_n)}$ is the same for $m_n = 0,1$, equation \eqref{eq:mainequality} reduces to
\begin{equation}
\begin{split}
\phi(\nu) \sin \bigg( \frac{\alpha_{(m_1,...,m_{n-1})}}{2} \bigg) + \phi(\nu) \sin \bigg( \frac{\alpha_{(m_1,...,m_{n-1})}}{2} -  \alpha_{(m_1,...,m_n)} \bigg) = \\
= (\phi(\nu_0) + \phi(\nu_1)) \sin \bigg( \frac{\alpha_{(m_1,...,m_n)}}{2} \bigg).
\end{split}
\end{equation}
Let us omit the second summand and instead consider a configuration for which holds the following equality
\begin{equation}\label{eq:simplifiedequality}
\phi(\nu) \sin \bigg( \frac{\alpha_{(m_1,...,m_{n-1})}}{2} \bigg) = (\phi(\nu_0) + \phi(\nu_1)) \sin \bigg( \frac{\alpha_{(m_1,...,m_{n-1},m_n)}}{2} \bigg).
\end{equation}
If we perform the construction so that condition \eqref{eq:simplifiedequality} is satisfied, we obtain a set with lower or equal measure in the limit (because at every step we remove a larger portion of the boundary). We denote the resulting sets $F'_n$ and $F'_\infty$. Therefore, if we show that under condition \eqref{eq:simplifiedequality} the $\mathcal{H}^1$-measure of $F'_\infty$ is positive, then also the $\mathcal{H}^1$-measure of $F_\infty$ is positive.

Notice that the angle of incidence of the line segment $l_{(m_1,...,m_{n-1})}$ to $\partial\Omega$ equals $\frac{\alpha_{(m_1,...,m_{n-1})}}{2}$. Therefore, for $i = 0,1$, the angle between the direction $\nu$ and $\nu_i$ is smaller or equal to $\frac{\alpha_{(m_1,...,m_{n-1})}}{2}$. Since a norm is locally Lipschitz, we have 
\begin{equation}
|\phi(\nu_i) - \phi(\nu)| \leq L \frac{\alpha_{(m_1,...,m_{n-1})}}{2}.
\end{equation}
We apply this to equation \eqref{eq:simplifiedequality} and get
\begin{equation}\label{eq:simplifiedequality2}
\begin{split}
\phi(\nu) \bigg| \sin \bigg( \frac{\alpha_{(m_1,...,m_{n-1})}}{2} \bigg) - 2 \sin \bigg( \frac{\alpha_{(m_1,...,m_n)}}{2} \bigg) \bigg| \leq  L \sin \bigg( \frac{\alpha_{(m_1,...,m_n})}{2} \bigg) \alpha_{(m_1,...,m_{n-1})}.
\end{split}
\end{equation}
We divide both sides by $\phi(\nu)$ and estimate the right hand side using the inequality $\sin(\alpha) \leq \alpha$. Since $\phi$ is bounded from below on $\mathbb{S}^1$, we get
\begin{equation}\label{eq:simplifiedequality3}
\begin{split}
\bigg| \sin \bigg( \frac{\alpha_{(m_1,...,m_{n-1})}}{2} \bigg) - 2 \sin \bigg( \frac{\alpha_{(m_1,...,m_n)}}{2} \bigg) \bigg| &\leq \frac{L}{2\phi(\nu)} \alpha_{(m_1,...,m_n)} \alpha_{(m_1,...,m_{n-1})}  \leq \\
&\leq { C_1(\phi)} \alpha_{(m_1,...,m_{n})} \alpha_{(m_1,...,m_{n-1})},
\end{split}
\end{equation}
where { $C_1(\phi)$} is a constant which depends only on the choice of $\phi$. We will infer from this an estimate on 
\begin{equation}
r_{(m_1,...,m_n)} := \frac{\alpha_{(m_1,...,m_{n-1})} - 2 \alpha_{(m_1,...,m_n)} }{2 \alpha_{(m_1,...,m_n)}} 
\end{equation}
in the following way. Since $\alpha \geq \sin(\alpha) \geq \alpha - \frac{\alpha^3}{3}$, we have
\begin{align*}
\frac{\alpha_{(m_1,...,m_{n-1})}}{2} - \alpha_{(m_1,...,m_n)} &\leq \sin \bigg( \frac{\alpha_{(m_1,...,m_{n-1})}}{2} \bigg) - 2 \sin \bigg( \frac{\alpha_{(m_1,...,m_n)}}{2} \bigg) + \frac{(\alpha_{(m_1,...,m_{n-1})})^3}{24} \\
&\leq \bigg| \sin \bigg( \frac{\alpha_{(m_1,...,m_{n-1})}}{2} \bigg) - 2 \sin \bigg( \frac{\alpha_{(m_1,...,m_n)}}{2} \bigg) \bigg| + \frac{(\alpha_{(m_1,...,m_{n-1})})^3}{24} \\
&\leq { C_1(\phi)} \alpha_{(m_1,...,m_{n})} \alpha_{(m_1,...,m_{n-1})} + \frac{(\alpha_{(m_1,...,m_{n-1})})^3}{24}.
\end{align*}
We divide this inequality by $\alpha_{(m_1,...,m_{n})}$ and get
\begin{equation}
r_{(m_1,...,m_n)} \leq { C_2(\phi)} \bigg( 1 + \frac{1}{24} \frac{(\alpha_{(m_1,...,m_{n-1})})^2}{\alpha_{(m_1,...,m_{n})}} \bigg) \alpha_{(m_1,...,m_{n-1})},
\end{equation}
{ where $C_2(\phi) = \max(C_1(\phi),1)$}. In order for \eqref{eq:simplifiedequality} to hold, for sufficiently small $\alpha_0$ we necessarily have that 
\begin{equation}\label{eq:boundfrombelowonquotient}
\frac{\alpha_{(m_1,...,m_{n})}}{\alpha_{(m_1,...,m_{n-1})}} \geq \frac14
\end{equation}
(and this quotient approaches $\frac12$ as $n \rightarrow \infty$). Therefore, the quotient $\frac{(\alpha_{(m_1,...,m_{n-1})})^2}{\alpha_{(m_1,...,m_{n})}}$ is bounded, and we get
\begin{equation}\label{eq:simplifiedequality4}
r_{(m_1,...,m_n)}  \leq { C_3(\phi)} \alpha_{(m_1,...,m_{n-1})}
\end{equation}
with a larger constant { $C_3(\phi)$}. Recall that in the construction at every step the angle $\alpha_{(m_1,...,m_n)}$ decreases by a factor of at least two. Without loss of generality, we may assume that $\alpha_0$ is small enough so that \eqref{eq:boundfrombelowonquotient} holds and that { $\alpha_0 < \frac{1}{2C_3(\phi)}$} (otherwise, we start the computation at a sufficiently large step of the construction). Then,
\begin{equation}\label{eq:boundforrn}
r_{(m_1,...,m_n)} \leq { C_3(\phi)} \alpha_{(m_1,...,m_{n-1})} < { C_3(\phi)} \frac{1}{2^{n-1}} \alpha_0 < \frac{1}{2^{n}}.
\end{equation}
Now, we compute the measure of $F'_\infty$. We have
\begin{align}
\mathcal{H}^1(F'_n) = \sum_{(m_1,...,m_n) \in \{ 0, 1 \}^n} \alpha_{(m_1,...,m_n)} &= \sum_{(m_1,...,m_{n-1}) \in \{ 0, 1 \}^{n-1}} \frac{\alpha_{(m_1,...,m_{n-1})}}{1 + r_{(m_1,...,m_{n-1})}} \\
&\geq \frac{1}{1 + \frac{1}{2^{n}}} \sum_{(m_1,...,m_{n-1}) \in \{ 0, 1 \}^{n-1}} \alpha_{(m_1,...,m_{n-1})}.
\end{align}
We now repeat the same argument $n-1$ times and obtain
\begin{equation}
\mathcal{H}^1(F'_n) \geq \frac{1}{1 + \frac{1}{2^{n}}} \sum_{(m_1,...,m_{n-1}) \in \{ 0, 1 \}^{n-1}} \alpha_{(m_1,...,m_{n-1})} \geq \frac{\alpha_0}{\Pi_{k=1}^n (1 + \frac{1}{2^k})}.
\end{equation}
The sum $\sum_{k=1}^\infty \frac{1}{2^k}$ is convergent, so the infinite product $\Pi_{k=1}^\infty (1 + \frac{1}{2^k})$ is also convergent, and we get
\begin{equation}
\mathcal{H}^1(F'_\infty) = \lim_{n \rightarrow \infty} \mathcal{H}^1(F_n) = \frac{\alpha_0}{\Pi_{k=1}^\infty (1 + \frac{1}{2^k})} > 0.
\end{equation}
Since $\mathcal{H}^1(F_\infty) \geq \mathcal{H}^1(F'_\infty)$, we get the desired result.
\end{proof}

\begin{remark}
In the case when $\phi$ is the Euclidean norm, the above proof is simpler, and by a simple modification we may obtain a sharper bound on the $\mathcal{H}^1$-measure of $F_\infty$. To this end, notice that for a circle of radius one the length of the chord corresponding to angle $\alpha$ equals $2 \sin(\frac{\alpha}{2})$. Therefore, equality \eqref{eq:mainequality} becomes
\begin{equation}\label{eq:isotropicdetailed}
2 \sin \bigg( \frac{\alpha_n}{2} \bigg) = \sin \bigg( \frac{\alpha_{n-1}}{2} \bigg) + \sin \bigg( \frac{\alpha_{n-1}}{2} - \alpha_n \bigg).
\end{equation}
Denote $r_n = \frac{\alpha_{n-1} - 2 \alpha_n}{\alpha_{n}}$, i.e. $r_n$ is the ratio of the length of the removed arc $\Gamma'_{(m_1,...,m_{n-1})}$ to the length of the remaining arc $\Gamma_{(m_1,...,m_{n-1},m_n)}$. We rewrite equation \eqref{eq:isotropicdetailed} as
\begin{equation}\label{eq:isotropicdetailed2}
2 \sin \bigg( \frac{\alpha_n}{2} \bigg) = \sin \bigg( \frac{r_n \alpha_{n}}{2} \bigg) + \sin \bigg( \alpha_n + \frac{r_n \alpha_{n}}{2} \bigg)
\end{equation}
We use the estimates $\alpha \geq \sin(\alpha) \geq \alpha - \frac{\alpha^3}{3}$ and $1 \geq \cos(\alpha) \geq 1 - \frac{\alpha^2}{2}$ and get
\begin{equation}\label{eq:isotropicdetailed3}
\begin{split}
\alpha &\geq 2 \sin \bigg( \frac{\alpha_n}{2} \bigg) = \sin \bigg( \frac{r_n \alpha_{n}}{2} \bigg) + \sin(\alpha_n) \cos \bigg( \frac{r_n \alpha_{n}}{2} \bigg) + \cos(\alpha_n) \sin \bigg( \frac{r_n \alpha_{n}}{2} \bigg) \geq \\
&\geq \frac{r_n \alpha_n}{2} - \frac{r_n^3 \alpha_n^3}{24} + \bigg( \alpha_n - \frac{\alpha_n^3}{3} \bigg) \bigg( 1 - \frac{r_n^2 \alpha_n^2}{8} \bigg) + \bigg( 1 - \frac{\alpha_n^2}{2} \bigg) \bigg( \frac{r_n \alpha_n}{2} - \frac{r_n^3 \alpha_n^3}{24} \bigg) = \\
&= \alpha_n - \frac{\alpha_n^3}{3} + r_n \alpha_n + O(r_n^2 \alpha_n^2).
\end{split}
\end{equation}
We reorganise the above inequality and get
\begin{equation}\label{eq:isotropicdetailed4}
\frac{\alpha_n^3}{3} \geq r_n \alpha_n (1 + O(r_n \alpha_n)).
\end{equation}
Since the sequence $\alpha_n$ is decreasing and $r_n < 1$, for sufficiently small $\alpha_0$ we have that $1 + O(r_n \alpha_n) \geq \frac{1}{2}$ for all $n$, so
\begin{equation}\label{eq:isotropicdetailed5}
r_n \leq \frac{2}{3} \alpha_n^2 < \frac{2}{3 \cdot 4^n} \alpha_0^2.
\end{equation}
Now, we compute the measure of $F_\infty$. We have
\begin{equation}
\mathcal{H}^1(F_n) = 2^n \alpha_n = 2^n \frac{\alpha_{n-1}}{2 + r_n} = 2^n \frac{\alpha_0}{\Pi_{k=1}^n (2 + r_k)} = \frac{\alpha_0}{\Pi_{k=1}^n (1 + \frac{r_k}{2})} \geq \frac{\alpha_0}{\Pi_{k=1}^n (1 + \frac{\alpha_0^2}{3 \cdot 4^k})}.
\end{equation}
Since the infinite product $\Pi_{k=1}^\infty (1 + \frac{\alpha_0^2}{3 \cdot 4^k})$ is convergent, we get that
\begin{equation}
\mathcal{H}^1(F_\infty) = \lim_{n \rightarrow \infty} \mathcal{H}^1(F_n) = \frac{\alpha_0}{\Pi_{k=1}^\infty (1 + \frac{\alpha_0^2}{3 \cdot 4^k})} > 0.
\end{equation}
\end{remark}

Now, we show that under condition \eqref{eq:mainequality} it is no longer true that the solution $\chi_{E_n}$ is unique, but we may extract a solution whose form is suitable for later considerations.

\begin{proposition}\label{prop:nonuniquesolutions}
Suppose that \eqref{eq:mainequality} holds. Then, at every step of the construction, the functions $\chi_{E_n}$ and $\chi_{E'_n}$ are solutions to problem \eqref{eq:problem} with boundary data $f_n = \chi_{F_n}$.
\end{proposition}

\begin{proof}
Instead of making the construction for the sequence $\alpha_n$, let us first choose a sequence $\alpha'_n$ in a different way, so that inequality \eqref{eq:maininequality} holds. Then, by Proposition \ref{prop:existenceofEn}, there exists a unique solution to problem \eqref{eq:problem}, which is a characteristic function of a set which we denote by $E_n(\alpha'_n)$. Notice that for any finite step of the construction, we may approximate $\alpha_n$ with a sequence $\alpha_{n,k}$ with this property. We pass to the limit $k \rightarrow \infty$ and get that $\chi_{E_n(\alpha_{n,k})} \rightarrow \chi_{E_n}$ in $L^1(\Omega)$, where $E_n = E_n(\alpha_n)$. By Theorem \ref{thm:miranda}, $\chi_{E_n}$ is a function of $\phi$-least gradient. Since it satisfies the trace condition (it is easy to see, since the number of points which are endpoints of line segments in $\partial E_n$ is finite), it is a solution to \eqref{eq:problem}. 

Now, we prove that $\chi_{E'_n}$ is also a solution to problem \eqref{eq:problem}. To this end, we show that it has the same $\phi$-total variation as $\chi_{E_n}$ (it is clear that the trace is correct). We write
\begin{equation}
\int_\Omega |D\chi_{E_t}|_\phi = \sum_{(m_1,...,m_n) \in \{ 0,1 \}^n } \| \ell_{(m_1,...,m_n)} \|_\phi.
\end{equation}
By assumption \eqref{eq:mainequality}, the right hand side is equal to
\begin{equation}
\sum_{(m_1,...,m_{n-1}) \in \{ 0,1 \}^{n-1} } \| \ell_{(m_1,...,m_{n-1})} \|_\phi + \sum_{(m_1,...,m_{n-1}) \in \{ 0,1 \}^{n-1} } \| \ell'_{(m_1,...,m_{n-1})} \|_\phi.
\end{equation}
We argue in the same manner on the first summand, and after $n-1$ steps we obtain that
\begin{equation}
\begin{split}
\int_\Omega |D\chi_{E_t}|_\phi = \sum_{(m_1,...,m_{n-1}) \in \{ 0,1 \}^{n-1} } &\| \ell'_{(m_1,...,m_{n-1})} \|_\phi + ... + \sum_{m_1 \in \{ 0,1 \}} \| \ell'_{(m_1)} \|_\phi + \\
+ &\| \overline{p_{(0,0)} p_{(1,1)}} \|_\phi + \| \overline{p_{(0,1)} p_{(1,0)}} \|_\phi = \int_\Omega |D\chi_{E'_t}|_\phi,
\end{split}
\end{equation}
because the boundary of $E'_t$ consists exactly of the line segments in this sum. Hence, $\chi_{E'_t}$ has the same $\phi$-total variation as $\chi_{E_t}$, so it is also a solution to problem \eqref{eq:problem}.
\end{proof}

\begin{proposition}\label{prop:existencelimit}
Suppose that for all $(m_1,...,m_{n-1}) \in \{ 0,1 \}^{n-1}$ we have
\begin{equation}\label{eq:equalitycase2}
h_\phi(m_1,...,m_{n-1}) = 0.
\end{equation}
Then, there exists a solution to problem \eqref{eq:problem} with boundary data $f =\chi_{F_\infty}$.
\end{proposition}

\begin{proof}
It follows from Proposition \ref{prop:positivemeasure} that $\mathcal{H}^1(F_\infty) > 0$. By Proposition \ref{prop:nonuniquesolutions}, the function $\chi_{E'_n}$ is a solution to \eqref{eq:problem} with boundary data $\chi_{F_n}$. We pass to the limit; since the sequence { $E'_n$} is decreasing, we get $\chi_{E'_n} \rightarrow \chi_{E'_\infty}$ { in $L^1(\Omega)$}, where $E'_\infty$ has positive Lebesgue measure. By Theorem \ref{thm:miranda}, $\chi_{E'_\infty}$ is a function of $\phi$-least gradient. To finish the proof, we only need to show that the trace of $\chi_{E'_\infty}$ equals $\chi_{F_\infty}$.

Notice that by Lemma \ref{lem:superlevelsets} the trace of $\chi_{E'_\infty}$ is again a characteristic function of some set $F \subset \partial\Omega$. We need to show that $F = F_\infty$. To this end, we use the pointwise characterisation of the trace operator: recall that for $\mathcal{H}^1$-almost all $x \in \partial\Omega$ we have 
$$ \lim_{r \rightarrow 0} \, \dashint_{B(x,r) \cap \Omega} |\chi_{E'_\infty} (y) - \chi_F(x)| \, dy = 0. $$ 
Now, take $x \in \partial\Omega \setminus F_\infty$ with this property. Then, by the construction of the set $F_\infty$, its complement is a countable union of arcs $\Gamma'_{(m_1,...,m_{n-1})}$ for some $n \in \mathbb{N}$ and $(m_1,...,m_{n-1}) \in \{ 0, 1 \}^{n-1}$, plus the two arcs in $\partial\Omega \backslash F_1$. But the set $\Delta'_{(m_1,...,m_{n-1})}$ is disjoint with $E'_n$, so it is also disjoint with $E'_\infty$ (a similar property holds for $\partial\Omega \backslash F_1$). Therefore, for all $x \in \partial\Omega \setminus F_\infty$, there exists a neighbourhood of $x$ in $\Omega$ such that $\chi_{E'_\infty} = 0$, so $\chi_F(x) = 0$ for $\mathcal{H}^1$-almost all $x \in \partial\Omega \setminus F_\infty$. Hence, $F \subset F_\infty$ up to a set of measure zero.

On the other hand, take $x \in F_\infty$. Then, for all $n \in \mathbb{N}$ we also have $x \in F_n$. Notice that in every step of construction the angle between $\partial \Omega$ and $l'_{(m_1,...,m_{n-1})}$ equals $\frac{\alpha_{(m_1,...,m_{n-1})}}{2}$, so (for $n \geq 2$) for every point of $F_n$ there is a cone $C_x$ of size $\frac{\pi}{2}$, bounded by two line segments which intersect $\partial\Omega$ at $x$ at angle $\frac{\pi}{4}$, which locally lies entirely in $E'_n$. We may take exactly the same cone for all $n \in \mathbb{N}$, so for sufficiently small $r$ (depending on $x$) we have that $C_x \cap B(x,r) \subset { E'_\infty}$. Hence,
$$\dashint_{B(x,r) \cap \Omega} { |\chi_{E'_\infty}(y)|} \, dy \geq \frac{1}{2} \, \dashint_{B(x,r) \cap C_x} { |\chi_{E'_\infty}(y)|} \, dy = \frac{1}{2}.$$
Since the mean integral in the pointwise definition of the trace is bounded from below, the trace of { $\chi_{E'_\infty}$} cannot be equal to zero for $\mathcal{H}^1$-almost all $x \in F_\infty$, so we have $\chi_F(x) = 1$ for $\mathcal{H}^1$-almost all $x \in F_\infty$. Hence, $F_\infty \subset F$ up to a set of measure zero, and the proof is concluded.
\end{proof}

Finally, we prove the main result of the paper. By the $1$-homogeneity of norms $\phi_1$ and $\phi_2$, we may equivalently consider their restrictions to $\mathbb{S}^1$, and we understand the assumption that $\phi_i$ is of class $C^2$ as $\phi_i |_{\mathbb{S}^1} \in C^2(\mathbb{S}^1)$ (obviously, it cannot be differentiable at zero). In the proof below, we treat $\nu$ and its variants as angles (i.e. parameters on $\mathbb{S}^1$) instead of vectors in $\mathbb{R}^2$, and all derivatives appearing in the proof are tangential derivatives along $\mathbb{S}^1$.

{\it Proof of Theorem \ref{thm:tracespace}.}
Since $\phi_1 \neq c\phi_2$, possibly after rescaling { the norm $\phi_2$} we can find a direction $\nu_0$ such that $\phi_1(\nu_0) = \phi_2(\nu_0)$ and $\phi''_1(\nu_0) > \phi''_2(\nu_0)$. Note that rescaling of a norm does not change the area-minimising sets. Then, there exists a neighbourhood $N \subset \mathbb{S}^1$ of $\nu_0$ such that
\begin{equation}\label{eq:estimatesonderivatives}
\phi''_1(\nu) \geq \inf_N \phi''_1 > \sup_N \phi''_2 \geq \phi''_2(\nu)
\end{equation}
for all $\nu \in N$. We may assume that it is small enough so that the angular coordinate does not have a jump in $N$. Furthermore, for $i = 1,2$ denote by $\omega_i$ the modulus of continuity of $\phi''_i$. Possibly making the neighbourhood $N$ smaller, we may require that 
\begin{equation}\label{eq:estimatesonderivativespart2}
\omega_i(\mbox{diam}(N)) < { \frac{1}{16}} (\inf_N \phi''_1 - \sup_N \phi''_2)
\end{equation}
for $i = 1,2$. Here, $\mbox{diam}(N) = \sup_{\alpha,\beta \in N} |\alpha - \beta|$. { Finally, since $\phi_1(\nu_0) = \phi_2(\nu_0)$ and both functions are continuous, again possibly making the neighbourhood $N$ smaller we may require that the ratio $\frac{\phi_2}{\phi_1}$ is arbitrarily close to $1$ on $N$, i.e. given $\varepsilon > 0$ we have
\begin{equation}\label{eq:estimatesonnorms}
\bigg| \frac{\phi_2(\nu)}{\phi_1(\nu)} - 1 \bigg| < \varepsilon
\end{equation}
for all $\nu \in N$. 

In the course of the proof, we will further rescale the norm $\phi_2$ in order to have $\phi_1(\nu) = \phi_2(\nu)$ for some given point $\nu \in N$ other than $\nu_0$. Since before rescaling we have $\phi_1(\nu_0) = \phi_2(\nu_0)$, from property \eqref{eq:estimatesonnorms} it follows that in order to obtain $\phi_1(\nu) = \phi_2(\nu)$ after rescaling we need to multiply $\phi_2$ by a constant $r \in (1-\varepsilon,1+\varepsilon)$. Therefore, if we chose the constant $\varepsilon > 0$ in \eqref{eq:estimatesonnorms} small enough, after rescaling $\phi_2$ so that $\phi_1(\nu) = \phi_2(\nu)$ the property \eqref{eq:estimatesonderivatives} remains true, and an estimate similar to the one in \eqref{eq:estimatesonderivativespart2} still holds: we have
\begin{equation}\label{eq:rescaledestimatesonderivativespart2}
\omega_i(\mbox{diam}(N)) < { \frac{1}{8}} (\inf_N \phi''_1 - \sup_N \phi''_2)
\end{equation}
for $i = 1,2$. Note that we first choose $\varepsilon > 0$ and then the neighbourhood $N$, which allows for the estimate \eqref{eq:rescaledestimatesonderivativespart2} to be independent on the choice of $\nu$.}

We construct the set $F_\infty$ with respect to $\phi_1$ for $\alpha_0$, which is small enough that the direction $\nu_{(m_1,...,m_n)}$ of any of the line segments $\ell_{(m_1,...,m_n)}$ lies in $N$. We require additionally that condition \eqref{eq:mainequality} holds, i.e. for all $(m_1,...,m_{n-1}) \in \{ 0 , 1 \}^{n-1}$ we have
\begin{equation}
h_{\phi_1}(m_1,...,m_{n-1}) = 0.
\end{equation}
We will show that for all $(m_1,...,m_{n-1}) \in \{ 0 , 1 \}^{n-1}$ we have
\begin{equation}\label{eq:inequalityforhphi2}
h_{\phi_2}(m_1,...,m_{n-1}) > 0.
\end{equation}
This will be achieved by a comparison with the corresponding value for $h_{\phi_1}$. { To simplify the argument, we rescale $\phi_2$ again, this time requiring that $\phi_1(\nu_{(m_1,...,m_{n-1})}) = \phi_2(\nu_{(m_1,...,m_{n-1})})$. Note that rescaling $\phi_2$ does not change the validity of inequality \eqref{eq:inequalityforhphi2}. Furthermore, as discussed in the previous paragraph, this rescaling does not change the validity of property \eqref{eq:estimatesonderivatives} and we have the estimate \eqref{eq:rescaledestimatesonderivativespart2}.}

We will compare the two numbers $h_{\phi_1}(m_1,...,m_{n-1})$ and $h_{\phi_2}(m_1,...,m_{n-1})$ in the following way: we fix a sequence $(m_1,...,m_{n-1}) \in \{ 0, 1 \}^{n-1}$ and we set $\alpha: = \alpha_{(m_1,...,m_n)}$ as a free variable. In other words, we fix the positions of points $p_{(m_1,...,m_{n-1},0,0)}$ and $p_{(m_1,...,m_{n-1},1,1)}$, but we do not yet fix the position of points $p_{(m_1,...,m_{n-1},0,1)}$ and $p_{(m_1,...,m_{n-1},1,0)}$, which is prescribed by the parameter $\alpha$. Denote $\overline{\alpha} := \alpha_{(m_1,...,m_{n-1})}$ and set 
\begin{equation}
g(\alpha) = h_{\phi_1}(m_1,...,m_{n-1})(\alpha) - h_{\phi_2}(m_1,...,m_{n-1})(\alpha).
\end{equation}
We will prove that $g(\alpha) < 0$ for all $\alpha \in (0,\frac{\overline{\alpha}}{2})$; this means that whenever $h_{\phi_1}(m_1,...,m_{n-1}) = 0$, we have that $h_{\phi_2}(m_1,...,m_{n-1}) > 0$.


To simplify the notation, in the computation below denote $\overline{\nu} : = \nu_{(m_1,...,m_{n-1})}$ and $\nu_\alpha^i := \nu_{(m_1,...,m_{n-1},i)}$ for $i = 0,1$. Let us again stress that in this proof we treat $\nu_\alpha^i$ and $\overline{\nu}$ as angles and not as vectors in $\mathbb{R}^2$. In particular, we have $\nu_\alpha^0 - \overline{\nu} = \frac{\overline{\alpha}}{2} - \frac{\alpha}{2}$ and $\nu_\alpha^1 - \overline{\nu} = \frac{\alpha}{2} - \frac{\overline{\alpha}}{2}$. Therefore, in the notation introduced above $g(\alpha)$ takes the form
\begin{equation}
\begin{split}
g(\alpha) = \bigg[ &\phi_1(\overline{\nu}) \sin \frac{\overline{\alpha}}{2} + \phi_1(\overline{\nu}) \sin \bigg( \frac{\overline{\alpha}}{2} - \alpha \bigg) - \phi_1(\nu_\alpha^0) \sin \frac{\alpha}{2} - \phi_1(\nu_\alpha^1) \sin \frac{\alpha}{2} \\
- &\phi_2(\overline{\nu}) \sin \frac{\overline{\alpha}}{2} - \phi_2(\overline{\nu}) \sin \bigg( \frac{\overline{\alpha}}{2} - \, \alpha \bigg) + \phi_2(\nu_\alpha^0) \sin \frac{\alpha}{2} + \phi_2(\nu_\alpha^1) \sin \frac{\alpha}{2}  \bigg] \\
= &\sin \frac{\alpha}{2} \bigg[- \phi_1(\nu_\alpha^0) - \phi_1(\nu_\alpha^1) + \phi_2(\nu_\alpha^0) + \phi_2(\nu_\alpha^1)  \bigg].
\end{split}
\end{equation}
We will prove the estimate for $g$ using a Taylor expansion up to the second order for $\phi_i$ around $\overline{\nu}$. For $i = 1,2$, we have
\begin{align}\label{eq:estimateofh1part1}
\phi_i(\nu_\alpha^0) = \phi_i(\overline{\nu}) + \phi'_i(\overline{\nu}) \cdot (\nu_\alpha^0 - \overline{\nu}) + \frac12 \phi''_{i}(\overline{\nu}) \cdot (\nu_\alpha^0 - \overline{\nu})^2 + R_i^0,
\end{align}
where
\begin{equation}\label{eq:estimateforr0}
|R_i^0| \leq \omega_i(\nu_\alpha^0 - \overline{\nu}) \cdot (\nu_\alpha^0 - \overline{\nu})^2 \leq \omega_i(\mbox{diam}(N)) \cdot (\nu_\alpha^0 - \overline{\nu})^2.
\end{equation}
Similarly, we have
\begin{align}\label{eq:estimateofh1part2}
\phi_i(\nu_\alpha^1) = \phi_i(\overline{\nu}) + \phi'_i(\overline{\nu}) \cdot (\nu_\alpha^1 - \overline{\nu}) + \frac12 \phi''_{i}(\overline{\nu}) \cdot (\nu_\alpha^1 - \overline{\nu})^2 + R_i^1,
\end{align}
where
\begin{equation}\label{eq:estimateforr1}
|R_i^1| \leq \omega_i(\nu_\alpha^1 - \overline{\nu}) \cdot (\nu_\alpha^1 - \overline{\nu})^2 \leq \omega_i(\mbox{diam}(N)) \cdot (\nu_\alpha^0 - \overline{\nu})^2.
\end{equation}
Now, we sum up equations \eqref{eq:estimateofh1part1} and \eqref{eq:estimateofh1part2}. Using the fact that $\nu_\alpha^0 - \overline{\nu} = \frac{\overline{\alpha}}{2} - \frac{\alpha}{2}$ and $\nu_\alpha^1 - \overline{\nu} = \frac{\alpha}{2} - \frac{\overline{\alpha}}{2}$, we see that the first-order terms cancel out and we get
\begin{equation}\label{eq:estimateofh1part3}
\phi_i(\nu_\alpha^0) + \phi_i(\nu_\alpha^1) = 2\phi_i(\overline{\nu}) + \phi''_i(\overline{\nu}) \cdot \bigg(\frac{\overline{\alpha}}{2} - \frac{\alpha}{2} \bigg)^2 + R_i^0 + R_i^1. 
\end{equation}
Now, we substract the estimate \eqref{eq:estimateofh1part3} for $i = 1$ from the same estimate for $i = 2$. Since $\phi_1(\overline{\nu}) = \phi_2(\overline{\nu})$, the zero-order terms cancel out and we get
\begin{equation}\label{eq:estimateofh1part4}
\phi_2(\nu_\alpha^0) + \phi_2(\nu_\alpha^1) - \phi_1(\nu_\alpha^0) - \phi_1(\nu_\alpha^1) = (\phi''_2(\overline{\nu}) - \phi''_1(\overline{\nu})) \cdot \bigg(\frac{\overline{\alpha}}{2} - \frac{\alpha}{2} \bigg)^2 + R_2^0 + R_2^1 - R_1^0 - R_1^1.
\end{equation}
Now, inequalities \eqref{eq:estimateforr0} and \eqref{eq:estimateforr1} imply that
\begin{equation}
|R_2^0 + R_2^1 - R_1^0 - R_1^1| \leq 2 (\omega_1(\mbox{diam}(N)) + \omega_2(\mbox{diam}(N))) \cdot \bigg(\frac{\overline{\alpha}}{2} - \frac{\alpha}{2} \bigg)^2.
\end{equation}
Because we chose the neighbourhood $N$ so that it satisfies properties \eqref{eq:estimatesonderivatives} and \eqref{eq:estimatesonderivativespart2}, { by estimate \eqref{eq:rescaledestimatesonderivativespart2}} we get
\begin{equation}
|R_2^0 + R_2^1 - R_1^0 - R_1^1| \leq \frac{1}{2} { (\phi''_1(\overline{\nu}) - \phi''_2(\overline{\nu}))} \bigg(\frac{\overline{\alpha}}{2} - \frac{\alpha}{2} \bigg)^2.
\end{equation}
{ Therefore, $\phi_2(\nu_\alpha^0) + \phi_2(\nu_\alpha^1) - \phi_1(\nu_\alpha^0) - \phi_1(\nu_\alpha^1)$ is negative because $\phi''_1(\overline{\nu}) > \phi''_2(\overline{\nu})$.} Since $g(\alpha)$ is this number multiplied by $\sin \frac{\alpha}{2}$, which is positive for all $\alpha \in (0,\frac{\overline{\alpha}}{2})$, we get that $g(\alpha) < 0$. This shows that $h_{\phi_1}(m_1,...,m_{n-1}) = 0$ and $h_{\phi_2}(m_1,...,m_{n-1}) > 0$. Hence, for boundary data $f = \chi_{F_\infty}$, by Proposition \ref{prop:nonexistence} there is no solution to problem \eqref{eq:problem} for $\phi_2$, but by Proposition \ref{prop:existencelimit} there is a solution to problem \eqref{eq:problem} for $\phi_1$, which finishes the proof. \qed

In case when the norms are not of class $C^2$, we can prove a related weaker result: we show that for a given strictly convex norm $\phi_1$, we can find a norm $\phi_2$ arbitrarily close to $\phi_1$ in such a way that the trace spaces do not coincide.

\begin{proposition}
Suppose that $\phi_1$ is a strictly convex norm. Then, there exists a strictly convex norm $\phi_2$ and a function $f \in L^\infty(\partial\Omega)$ such that there exists a solution to \eqref{eq:problem} for $\phi_1$, but there is no solution to \eqref{eq:problem} for $\phi_2$. Moreover, we can require that $\phi_2$ be arbitrarily close to $\phi_1$ in the supremum norm on $\mathbb{S}^1$.
\end{proposition}

\begin{proof}
Construct the set $F_\infty \subset \partial\Omega$ in the first quadrant of the coordinate plane, i.e. part of $\partial B(0,1)$ corresponding to $\alpha \in [0,\frac{\pi}{2}]$. We do this in such a way that
\begin{equation}
h_{\phi_1}(m_1,...,m_{n-1}) = 0 \text{ for all } (m_1,...,m_{n-1}) \in \{ 0,1 \}^{n-1}.
\end{equation}
In particular, by Proposition \ref{prop:positivemeasure} we have $\mathcal{H}^1(F_\infty) > 0$. Since $F_\infty$ lies in the first quadrant, for all $n \in \mathbb{N}$ all line segments between points in $F_n$ are area-minimising for the $l_1$ norm, so we may define $h_{l_1}(m_1,...,m_{n-1})$ by formula \eqref{eq:definitionofh} even though the norm is not strictly convex. We notice that since the $l_1$-length of $\ell_{(m_1,...,m_{n-1})}$ is greater than the sum of the $l_1$-lengths of $\ell_{(m_1,...,m_{n-1},0)}$ and $\ell_{(m_1,...,m_{n-1},1)}$, we have
\begin{equation}
h_{l_1}(m_1,...,m_{n-1}) > 0 \text{ for all } (m_1,...,m_{n-1}) \in \{ 0,1 \}^{n-1}.
\end{equation}
Finally, take $\phi_2 = \phi_1 + \frac{1}{n} l_1$. Then, it satisfies
\begin{equation}
h_{\phi_2}(m_1,...,m_{n-1}) = h_{\phi_1}(m_1,...,m_{n-1}) + \frac{1}{n} h_{l_1}(m_1,...,m_{n-1}) > 0
\end{equation}
for all $(m_1,...,m_{n-1}) \in \{ 0,1 \}^{n-1}$. Therefore, by Proposition \ref{prop:existencelimit}, there exists a solution to problem \eqref{eq:problem} with boundary data $f =\chi_{F_\infty}$ for $\phi_1$, but by Proposition \ref{prop:nonexistence} there is no solution for $\phi_2$, and the two norms can be arbitrarily close in the supremum norm on $\mathbb{S}^1$.
\end{proof}

Finally, let us briefly comment on the non-strictly convex case. The main difference with respect to the strictly convex case is that in the construction of the set $F_\infty$ in the neighbourhood of a given direction, due to the fact that the triangle inequality is not strict, it might be impossible to enforce the condition $h_\phi(m_1,...,m_{n-1}) = 0$ for all $\{ m_1, ..., m_{n-1} \} \in \{ 0, 1 \}^{n-1}$. Below, we give an example for the $l_1$ norm.

\begin{proposition}
Suppose that $\phi = l_1$. Suppose that $F_\infty \subset \partial\Omega$ lies in the first quadrant of the coordinate plane and $\mathcal{H}^1(F_\infty) > 0$. Then, there is no solution to problem \eqref{eq:problem} with boundary data $f =\chi_{F_\infty}$.
\end{proposition}

\begin{proof}
We perform a similar argument as in the proof of Proposition \ref{prop:nonexistence}. As in the proof of the previous Proposition, since $F_\infty$ lies in the first quadrant of the coordinate plane, for all $n \in \mathbb{N}$ we may define $h_{l_1}(m_1,...,m_{n-1})$ by formula \eqref{eq:definitionofh} even though the norm is not strictly convex and $h_{l_1}(m_1,...,m_{n-1}) > 0$ for all $(m_1,...,m_{n-1}) \in \{ 0,1 \}^{n-1}$.

Now, we slightly alter Proposition \ref{prop:existenceofEn} in the following way: in the construction of $E_n$, replace each of the sets $\Delta_{(m_1,...,m_n)}$ (whose boundary consists of $\Gamma_{(m_1,...,m_n)}$ and $\ell_{(m_1,...,m_n)}$) with the set $\Delta^1_{(m_1,...,m_n)}$, whose boundary consists of $\Gamma_{(m_1,...,m_n)}$ and two line segments, one vertical and one horizontal. We denote the resulting set by $E_n^{1}$. Then, a careful examination of the proof of Proposition \ref{prop:existenceofEn} yields that while solutions to \eqref{eq:problem} for boundary data $f_n = \chi_{F_n}$ are no longer unique, $\chi_{E_n^1}$ is a solution and any other solution $u_n$ satisfies $u_n \leq \chi_{E_n^1}$. But then, since $\bigcap_{n=0}^\infty E_n^1 = \emptyset$, we make the same argument as in the proof of Proposition \ref{prop:nonexistence} to conclude that there is no solution to \eqref{eq:problem} for boundary data $f =\chi_{F_\infty}$.
\end{proof}

{\flushleft \bf Funding.} This work was partially supported by the DFG-FWF project FR 4083/3-1/I4354, by the OeAD-WTZ project CZ 01/2021, and by the project 2017/27/N/ST1/02418 funded by the National Science Centre, Poland.

{\flushleft \bf Conflict of interest.} On behalf of all authors, the corresponding author states that there is no conflict of interest.

{\flushleft \bf Data availability statement.} Data sharing not applicable to this article as no datasets were generated or analysed during the current study.

\end{document}